\documentclass[a4paper]{amsart}
\usepackage{amsfonts}
\usepackage{amssymb,verbatim}
\usepackage{amsmath}
\usepackage{amsthm}
\usepackage[english]{babel}
\usepackage[latin1]{inputenc}

\usepackage{pdfsync}
\usepackage{tabularx}
\usepackage[active]{srcltx}
\usepackage{array}
\usepackage{enumerate}
\usepackage{xcolor}
\usepackage{graphicx}
\usepackage{fancyhdr}

\newtheorem{theorem}{Theorem}[section]

\newtheorem{proposition}[theorem]{Proposition}
\newtheorem{defi}[theorem]{Definition}
\newtheorem{remark}[theorem]{Remark}

\renewcommand \parallel {/\kern-3pt/}

\newcommand \C {\mathbb C}
\newcommand \N {\mathbb N}
\newcommand \Z {\mathbb Z}

\newcommand \ben {\begin{enumerate}}
\newcommand \een {\end{enumerate}}

\newcommand \z{\mathfrak{z}}
\newcommand{\n}{\mathfrak{n}}
\renewcommand{\v}{\mathfrak{v}}

\newcommand{\la}{\lambda}

\numberwithin{equation}{section}

\begin{document}

\title[Some remarks on graded nilpotent Lie algebras and the TRC]
{Some remarks on graded nilpotent Lie algebras and the Toral Rank Conjecture}

\author{Guillermo Ames}
\address{Universidad Tecnol\'{o}gica Nacional, Facultad Regional C\'{o}rdoba, Argentina.}
\email{lames@scdt.frc.utn.edu.ar}
\thanks{The first author was partially supported by UTN-FRC and SECYT-UNC grants.}

\author{Leandro Cagliero}
\address{CIEM-CONICET, FAMAF-Universidad Nacional de C\'ordoba, C\'ordoba, Argentina.}
\email{cagliero@famaf.unc.edu.ar}
\thanks{The second author was supported in part by CONICET, CIUNSa and SECYT-UNC grants.}

\author{M\'{o}nica Cruz}
\address{Facultad de Ciencias Exactas, Universidad Nacional de Salta, Argentina.}
\email{monicanancy@gmail.com}
\thanks{The third author was supported in part by a CIUNSa grant}



\begin{abstract}
If
$\mathfrak{n}$ is a $\Z^d_+$-graded nilpotent finite dimensional Lie algebra over a field of characteristic zero,
it is well known that
$
\dim H^{\ast }(\mathfrak{n})\geq L(p)
$
where $p$ is the polynomial associated to the grading and
$L(p)$ is the sum of the absolute values of the coefficients of $p$.
From this result Deninger and Singhof derived the Toral Rank Conjecture (TRC)
for 2-step nilpotent Lie algebras.
An algebraic version of the TRC states that
$
\dim H^{\ast }(\mathfrak{n})\geq 2^{\dim (\mathfrak{z)}}
$
for any finite dimensional Lie algebra $\n$ with center $\z$.

The TRC is more that 25 years old and remains open
even for $\Z^d_+$-graded 3-step nilpotent Lie algebras.
Investigating to what extent the above bound for $\dim H^{\ast }(\mathfrak{n})$
could help to prove the TRC in this case, we considered the following two questions
regarding a nilpotent Lie algebra $\n$ with center $\z$:
(A) If $\n$ admits a $\Z_+^d$-grading
 $\n=\bigoplus_{\alpha\in\Z_+^d} \n_\alpha$,
such that its associated polynomial $p$ satisfies $L(p)>2^{\dim\z}$,
does it admit a grading
$\mathfrak{n}=\mathfrak{n}'_{1}\oplus \mathfrak{n}'_{2}\oplus \dots\oplus \mathfrak{n}'_{k}$
such that its associated polynomial $p'$ satisfies $L(p')>2^{\dim\z}$?
(B)  If $\n$ is $r$-step nilpotent admitting a grading
$\mathfrak{n}=\mathfrak{n}_{1}\oplus \mathfrak{n}_{2}\oplus \dots\oplus \mathfrak{n}_{k}$
such that its associated polynomial $p$ satisfies $L(p)>2^{\dim\z}$,
does it admit a grading
$\mathfrak{n}=\mathfrak{n}'_{1}\oplus \mathfrak{n}'_{2}\oplus \dots\oplus \mathfrak{n}'_{r}$
such that its associated polynomial $p'$ satisfies $L(p')>2^{\dim\z}$?
In this paper we show that the answer to (A) is yes, but the answer to (B) is no.
\end{abstract}

\maketitle

\section{Introduction}


Let $\Z_+^d=\{(\alpha_1,\dots,\alpha_n)\ne 0:\alpha_i\in\Z_{\ge0}\}$.
Given a $\Z_+^d$-graded finite dimensional nilpotent Lie algebra over a field of characteristic zero
\[
\n=\bigoplus_{\alpha\in\Z_+^d} \n_\alpha,
\]
(that is $[\n_\alpha, \n_\beta]\subset\n_{\alpha+\beta}$),
Deninger and Singhof \cite{DS} considered
the \emph{polynomial associated} to the grading
\begin{align*}
 p(x_1,\dots,x_d)&=\prod_{\alpha\in\Z_+^d}  (1-x_1^{\alpha_1}\cdots x_d^{\alpha_d})^{d_\alpha},\qquad
d_\alpha=\dim \n_\alpha, \\
                 &=\sum_{\alpha\in\Z_+^d} a_\alpha x_1^{\alpha_1}\cdots x_d^{\alpha_d}\in\Z[x_1,\dots,x_d];
\end{align*}
and proved that
\begin{equation}\label{eq.DS}
 \dim H^{\ast }(\mathfrak{n})\geq L(p)
\end{equation}
where
$
L(p)=\sum_{\alpha\in\Z_+^d} |a_\alpha|
$
is the \emph{length} of $p$ (see also \cite[Theorem 1]{Ma}).

If $\n=\v\oplus\z$ is a 2-step nilpotent Lie algebra,
with center $\z=[\n,\n]$, then $\n_1=\v$ and $\n_2=\z$
defines a $\Z$-grading on $\n$.
Deninger and Singhof considered in \cite{DS} this extremely
particular instance of their result
to obtain the Toral Rank Conjecture (TRC) for 2-step nilpotent
Lie algebras.

Recall that the TRC was formulated by Halperin \cite{Hal}
more than 25 years ago and an algebraic version of it is the following:

\medskip
\noindent
\textbf{TRC.}
If $\n$ is a finite dimensional nilpotent Lie algebra with center $\z$, then
\begin{equation*}
\dim H^{\ast }(\mathfrak{n})\geq 2^{\dim \z}.
\end{equation*}

This conjecture has a topological origin:
the toral rank $r(X)$ of a differentiable manifold $X$ is
the dimension of the greatest torus acting freely on $X$.
Originally, the TRC states that the homology of the
manifold $X$ has dimension greater than or equal to $2^{r(X)}$.
It follows from a theorem of Nomizu \cite{N} that, for compact nilmanifolds,
the original TRC would follow from the algebraic version stated above.

This conjecture remains open in general (see \cite{PT} or \cite{Y} for more information).
Although a very particular instance of
\eqref{eq.DS} implies the TRC for 2-step nilpotent Lie algebras,
the TRC remains open even for $\Z_+^d$-graded 3-step nilpotent Lie algebras
and it is natural to ask to what extent the result of Deninger and Singhof
could help to prove the TRC in this case.
It might be worth mentioning that sometimes \eqref{eq.DS} is pretty accurate,
for instance
\[
 \dim H^{\ast }(\mathfrak{n})=L(p)
\]
for Heisenberg Lie algebras (of arbitrary dimensions) and for
every nilpotent Lie algebra of dim $\le 6$,
(when they are given an appropriate grading \cite[\S1.4.4 and \S1.6.1]{Ma}).

\medskip

Given a $\Z_+^d$-graded nilpotent Lie algebra,
it is in general very difficult to consider, in the context of \eqref{eq.DS},
all its possible $\Z_+^d$-gradings, even for the 3-step case.
Thus we centered our attention to the following two questions
regarding a finite dimensional nilpotent Lie algebra $\n$ with center $\z$:
\begin{enumerate}[(A)]
 \item If $\n$ admits a $\Z_+^d$-grading
 $\n=\bigoplus_{\alpha\in\Z_+^d} \n_\alpha$,
such that its associated polynomial $p$ satisfies $L(p)>2^{\dim\z}$,
does it admit a grading
$\mathfrak{n}=\mathfrak{n}'_{1}\oplus \mathfrak{n}'_{2}\oplus \dots\oplus \mathfrak{n}'_{k}$
such that its associated polynomial $p'$ satisfies $L(p')>2^{\dim\z}$?

 \item If $\n$ is $r$-step nilpotent admitting a grading
$\mathfrak{n}=\mathfrak{n}_{1}\oplus \mathfrak{n}_{2}\oplus \dots\oplus \mathfrak{n}_{k}$
such that its associated polynomial $p$ satisfies $L(p)>2^{\dim\z}$,
does it admit a grading
$\mathfrak{n}=\mathfrak{n}'_{1}\oplus \mathfrak{n}'_{2}\oplus \dots\oplus \mathfrak{n}'_{r}$
such that its associated polynomial $p'$ satisfies $L(p')>2^{\dim\z}$?
\end{enumerate}

In this short paper we show that the answer to (A) is yes but the answer to (B) is no, giving an example for $r=3$.

We do not know an example of a $\Z_+$-graded nilpotent Lie algebra satisfying $L(p)<2^{\dim\z}$
for all its possible gradings.

\section{The answer to (A) is yes}
\begin{proposition}
If $\n=\bigoplus_{\alpha\in\Z_+^d} \n_\alpha$ is $\Z_+^d$-graded, with associated polynomial $p$,
then there exists a grading
$\mathfrak{n}=\mathfrak{n}'_{1}\oplus \mathfrak{n}'_{2}\oplus \dots\oplus \mathfrak{n}'_{k}$
such that its associated polynomial $p'$ satisfies $L(p')=L(p)$.
\end{proposition}

\begin{proof}
If $d=1$ there is nothing to prove.
We now assume that the proposition is proved for
$\Z_+^{d-1}$-graded  Lie algebras
and let
$
\n=\bigoplus_{\alpha\in\Z_+^d} \n_\alpha,
$
be a $\Z_+^d$-graded nilpotent Lie algebra
with associated polynomial $p$.
Let us fix $m\in\N$ and consider the $\Z_+^{d-1}$-grading of $\n$ defined by
\[
\n=\bigoplus_{\beta\in\Z_+^{d-1}} \n_{\beta}
\]
where $\n_{\beta}=\bigoplus_{t=0}^{\left\lfloor\frac{\beta_1}{m}\right\rfloor}\n_{(\beta_1-mt,\beta_2,\dots,\beta_{d-1},t)}$.
It is straightforward to see that, if $p_m$ is the polynomial associated to this $\Z_+^{d-1}$-grading, then
\[
p_m(x_1,\dots,x_{d-1})=p(x_1,\dots,x_{d-1},x_1^m).
\]
It is clear that
\[
L(p_m)=\sum_{\beta\in\Z_+^{d-1}}
\Big|\sum_{t=0}^{\left\lfloor\frac{\beta_1}{m}\right\rfloor} a_{(\beta_1-mt,\beta_2,\dots,\beta_{d-1},t)}\Big|
\]
and, if
$
 m>\max\{\alpha_1:a_{(\alpha_1,\dots,\alpha_{d})}\ne0\},
$
then $L(p_m)=L(p)$, since $a_{(\beta_1-mt,\dots,\beta_{d-1},t)}\ne 0$ only if
$t=\left\lfloor\frac{\beta_1}{m}\right\rfloor$.

By the induction hypothesis there exits a grading
$\mathfrak{n}=\mathfrak{n}'_{1}\oplus \mathfrak{n}'_{2}\oplus \dots\oplus \mathfrak{n}'_{k}$
such that its associated polynomial $p'$ satisfies $L(p')=L(p_m)=L(p)$.
This completes the induction step.
\end{proof}

\section{The answer to (B) is no}
In this section we construct a family $\n=\n(n)$, $n\in\mathbb{N}$, of graded 3-step nilpotent Lie algebras
with center $\z$ 
such that $L(p)<2^{\dim\z}$ for all gradings $\n(n)=\n_1\oplus\n_2\oplus\n_3$ of $\n(n)$, for all $n\ge17$,
but admitting a $\Z$-grading whose associated polynomial has length greater that $2^{\dim\z}$.

\subsection{Definition of the family $\n(n)$}
In what follows, if $A$ is a set, $\langle A\rangle$ will denote the free vector space with $A$ as a basis.
For each positive integer $n$, let
\[
\begin{matrix}
E_n=\langle\{e_i:i=1,\dots,n\}\rangle, & U_n=\langle\{u_i:i=1,\dots,n\}\rangle, \\[2mm]
X_n=\langle\{x_i:i=1,\dots,n\}\rangle, & Y_n=\langle\{y_i:i=1,\dots,n\}\rangle.
\end{matrix}
\]

Since $n$ will be fixed most of the time, we will
use $E$, $U$, $X$ and $Y$ to denote the spaces $E_n$, $U_n$, $X_n$ e $Y_n$.
We will define on the vector space
\[
\mathfrak{n}=\n(n)=\n_1\oplus\n_2\oplus\n_3
\]
where
\begin{align*}
\mathfrak{n}_{1}&=E\oplus\langle\{ a, b, x\}\rangle, \\
\mathfrak{n}_{2}&=\langle\{  u, y\}\rangle \oplus \Lambda ^{2} E
			  \oplus \langle\{ c\} \rangle\oplus X, \\
\mathfrak{n}_{3}&=U \oplus Y \oplus \langle\{f,h\}\rangle,
\end{align*}
and let
\begin{align*}
\mathfrak{B}_1 &=\{e_{1},e_{2},..,e_{n},a,b,x\}, \\
\mathfrak{B}_2 &=\{u,y\}\cup \{e_{i}\wedge e_{j}:\,1\le i<j\le n\}
		      \cup\{c\}\cup\{x_{i}:\,1\le i\le n\}, \\
\mathfrak{B}_3 &=\{u_1,\dots,u_n,y_1,\dots,y_n,f,h\}.
\end{align*}
be ordered basis of $\mathfrak{n}_{1}$, $\mathfrak{n}_{2}$ and $\mathfrak{n}_{3}$
respectively (we choose the lexicographic order for $e_i\wedge e_j$).
Now
\[\mathfrak{B} =\mathfrak{B}_1 \cup\mathfrak{B}_2 \cup\mathfrak{B}_3 \]
is an ordered basis of $\n$.
It is clear that
\begin{align*}
d_{1}&=\dim (\mathfrak{n}_{1})=n+3, \\
d_{2}&=\dim (\mathfrak{n}_{2})=\frac{n(n+1)}{2}+3, \\
d_{3}&=\dim (\mathfrak{n}_{3})=2(n+1).
\end{align*}

We now define the Lie bracket of $\n$ in terms of this basis as shown
in the following table:
{\scriptsize{
\begin{equation*}
\setlength{\extrarowheight}{5pt}
\begin{tabular}{|c||c|c|c|c|c|c||c|c|c|c|c||c|c|c|c|}
\hline
$\!\![{\mathfrak{n},\mathfrak{n}}]\!\!$ & ${ e}_{1}$ & ${ \cdots }$ & ${ e}_{n}$ & ${ a}$ & ${ b}$ & ${ x}$ & ${ u}$ & $
{ y}$ &$\!\!{ \Lambda }^{2}{ E}\!\!$ & ${ c}$ & ${ X}$ & ${U}$ & ${Y}$ & ${ f}$ & ${ h}$ \\ \hline \\[-5mm] \hline
${ e}_{1}$ & ${ 0}$ & ${ \cdots }$ & $\!\!\!{ e}_{1}{ \wedge e}_{n}\!\!\!$ & ${ 0}$ & ${ 0}$ & ${ x}_{1}$ & ${ u}_{1}$ & ${ y}_{1}$
& ${ 0}$ & ${ 0}$ & ${ 0}$ & ${ 0}$ & ${ 0}$ & ${ 0}$ & ${ 0}$ \\ \hline
${ \vdots }$ & ${ \vdots }$ & ${ \ddots }$ & ${ \vdots }$ & ${ \vdots }$ & ${ \vdots }$ & ${ \vdots }$ & ${ \vdots }$ & ${ \vdots }$ & ${ \vdots }$ & ${ \vdots }$ & ${ \vdots }$ & ${ \vdots }$ & ${ \vdots }$ & ${ \vdots }$ & ${ \vdots }$ \\ \hline
${ e}_{n}$ & $\!\!\!{ e}_{n}{ \wedge e}_{1}\!\!\!$ & ${ \cdots }$ & ${ 0}$ & ${ 0}$ & ${ 0}$ & ${ x}_{n}$ & ${ u}
_{n}$ & ${ y}_{n}$ & ${ 0}$ & ${ 0}$ & ${ 0}$ & ${ 0}$ & ${ 0}$ & ${ 0}$ & ${ 0}$ \\ \hline
${ a}$ & ${ 0}$ & ${ \cdots }$ & ${ 0}$ & ${ 0}$ & ${ c}$ & ${ 0}$ & ${ 0}$ & ${ f}$ & ${ 0}$
& ${ h}$ & ${ 0}$ & ${ 0}$ & ${ 0}$ & ${ 0}$ & ${ 0}$ \\ \hline
${ b}$ & ${ 0}$ & ${ \cdots }$ & ${ 0}$ & ${ -c}$ & ${ 0}$ & ${ 0}$ & ${ h}$ & ${ h}$ & ${ 0}$
& ${ 0}$ & ${ 0}$ & ${ 0}$ & ${ 0}$ & ${ 0}$ & ${ 0}$ \\ \hline
${ x}$ & ${ -x}_{1}$ & ${ \cdots }$ & ${ -x}_{n}$ & ${ 0}$ & ${ 0}$ & ${ 0}$ & ${ f}$ & ${ h}$ & ${ 0}$ & ${ 0}$ & ${ 0}$ & ${ 0}$ & ${ 0}$ & $
{ 0}$ & ${ 0}$ \\ \hline \\[-5mm] \hline
${ u}$ & ${ -u}_{1}$ & ${ \cdots }$ & ${ -u}_{n}$ & ${ 0}$ & ${ -h}$ & ${ -f}$ & ${ 0}$ & ${ 0}$ & ${ 0}$ & ${ 0}$ & ${ 0}$ & ${ 0}$ & ${ 0}$ & $
{ 0}$ & ${ 0}$ \\ \hline
${ y}$ & ${ -y}_{1}$ & ${ \cdots }$ & ${ -y}_{n}$ & ${ -f}$ & ${ -h}$ & ${ -h}$ & ${ 0}$ & ${ 0}$ &
${ 0}$ & ${ 0}$ & ${ 0}$ & ${ 0}$ & ${ 0}$ & ${ 0}$ & ${ 0}$ \\ \hline
$\!\!{ \Lambda }^{2}{ E}\!\!$ & ${ 0}$ & ${ \cdots }$ & ${ 0}$ & ${ 0}$ & ${ 0}$ & ${ 0}$ & ${ 0}$ & ${ 0}$ & ${ 0}$ & ${ 0}$ & ${ 0}$ & ${ 0}$ & $
{ 0}$ & ${ 0}$ & ${ 0}$ \\ \hline
${ c}$ & ${ 0}$ & ${ \cdots }$ & ${ 0}$ & ${ -h}$ & ${ 0}$ & ${ 0}$ & ${ 0}$ & ${ 0}$ & ${ 0}$
& ${ 0}$ & ${ 0}$ & ${ 0}$ & ${ 0}$ & ${ 0}$ &${ 0}$ \\ \hline
${X}$ & ${ 0}$ & ${ \cdots }$ & ${ 0}$ & ${ 0}$ & ${ 0}$ & ${ 0}$ & ${ 0}$ & ${ 0}$ & ${ 0}$
& ${ 0}$ & ${ 0}$ & ${ 0}$ & ${ 0}$ & ${ 0}$ &${ 0}$ \\ \hline \\[-5mm] \hline
${ U}$ & ${ 0}$ & ${ \cdots }$ & ${ 0}$ & ${ 0}$ & ${ 0}$ & ${ 0}$ & ${ 0}$ & ${ 0}$ & ${ 0}$
& ${ 0}$ & ${ 0}$ & ${ 0}$ & ${ 0}$ & ${ 0}$ &${ 0}$ \\ \hline
${Y}$ & ${ 0}$ & ${ \cdots }$ & ${ 0}$ & ${ 0}$ & ${ 0}$ & ${ 0}$ & ${ 0}$ & ${ 0}$ & ${ 0}$
& ${ 0}$ & ${ 0}$ & ${ 0}$ & ${ 0}$ & ${ 0}$ & ${ 0}$ \\ \hline
${ f}$ & ${ 0}$ & ${ \cdots }$ & ${ 0}$ & ${ 0}$ & ${ 0}$ & ${ 0}$ & ${ 0}$ & ${ 0}$ & ${ 0}$
& ${ 0}$ & ${ 0}$ & ${ 0}$ & ${ 0}$ & ${ 0}$ & ${ 0}$ \\ \hline
${ h}$ & ${ 0}$ & ${ \cdots }$ & ${ 0}$ & ${ 0}$ & ${ 0}$ & ${ 0}$ & ${ 0}$ & ${ 0}$ & ${ 0}$
& ${ 0}$ & ${ 0}$ & ${ 0}$ & ${ 0}$ & ${ 0}$ &${ 0}$ \\ \hline
\end{tabular}
\end{equation*}
}}

\begin{remark}\label{rk.independence of basis}
Suppose that we change the basis $B=\{e_i\}$ of $E$ by $B'=\{e_i'\}$, and we
define accordingly $x_i'=[e_i',x]$, $y_i'=[e_i',y]$, $u_i'=[e_i',u]$.
If we consider the following new ordered basis of $\n_1$, $\n_2$ and $\n_3$
\begin{align*}
\mathfrak{B}_1'&=\{e_{1}',e_{2}',..,e_{n}',a,b,x\} \\
\mathfrak{B}_2'&=\{u,y\}\cup \{e_{i}'\wedge e_{j}':\,1\le i<j\le n\}\cup \{c\}\cup \{x_{i}':\,1\le i\le n\} \\
\mathfrak{B}_3'&=\{u_1',\dots,u_n',y_1',\dots,y_n',f,h\}
\end{align*}
then the above bracket-table looks the same for the new basis
$\mathfrak{B}'=\mathfrak{B}_1' \cup\mathfrak{B}_2' \cup\mathfrak{B}_3'$ of $\n$.
From now on, we will think of $\mathfrak{B}$ as a map that assigns a basis $B'$ of $E$
to  $\mathfrak{B}(B')=\mathfrak{B}'$.
\end{remark}

\begin{proposition}\label{prop.dim}
For every $n\in\mathbb N$, $\n(n)$ is a graded 3-step nilpotent Lie algebra with
\[
 d_{1}=n+3,\quad d_{2}=\frac{n(n+1)}2+3,\quad d_{3}=2n+2,
\]\[
 d_{2}^0=2,\quad d_{2}^1=1,\quad z_{2}=\frac{n(n+1)}{2},\quad z=\frac{(n+4)(n+1)}{2}.
\]
\end{proposition}

\begin{proof}The only basis elements $t,v,w$ such that $[ [t,v],w]\ne 0$ are $t=b$ and $v=w=a$.
Therefore, the Jacobi's identity is trivially satisfied in $\mathfrak{n}$.
\end{proof}


\subsection{Derivations of $\n(n)$}
Let $\text{Der}(\mathfrak{n})$ be the Lie algebra of derivations of $\n$.
In this subsection we will describe some properties of the matrices corresponding to
elements in $\text{Der}(\mathfrak{n})$ associated to a basis $\mathfrak{B}(B)$, where $B$ is a basis of $E$.

\begin{defi}
We denote by $\text{Der}(\mathfrak{n})_{0}$ the subalgebra of derivations
$D$ such that
\begin{itemize}
\item[--] $D(E)\subset E$ and
\item[--] $D(a)=D(b)=D(x)=D(u)=D(y)=0$.
\end{itemize}
\end{defi}

It is clear that there is Lie algebra isomorphism
\begin{align*}
 {\mathfrak{gl}}(E)&\to\text{Der}(\mathfrak{n})_{0} \\
A&\mapsto D_A.
\end{align*}
The matrix $[D_A]_{\mathfrak{B}(B)}$ of $D_A$ in the basis $\mathfrak{B}(B)$ is block-diagonal,
where the blocks corresponding to each subspace are described by the following table:
\medskip
\begin{equation}\label{d0}
\begin{tabular}{|cccc|ccccc|cccc|}
\hline
$E$ & $a$ & $b$ & $x$ & $u$ & $y$ & $\Lambda^{2}E$ & $c$ & $X$ &
$U$ & $Y$ & $f$ & $h$\rule[-1mm]{0mm}{5mm} \\
\hline
$[A]_B$ & 0 & 0 & 0& 0 & 0 & $[\Lambda^{2}A]_{\Lambda^{2}B}$ & 0 & $[A]_B$ &$[A]_B$ &$[A]_B$ & 0& 0
\rule[-1mm]{0mm}{5mm} \\
\hline
\end{tabular}
\end{equation}
\medskip

\noindent
Here, $\Lambda^{2}B=\{e_{i}\wedge e_{j}:\,1\le i<j\le n\}$
and
$\Lambda^{2}A$ is the linear map on $\Lambda^2 E$ defined by
$\Lambda^{2}A(e_{i}\wedge e_{j})=A(e_{i})\wedge e_{j}+e_{i}\wedge A(e_{j})$.
\begin{defi}
We denote by $\text{Der}(\mathfrak{n})_{1}$ the set of derivations $D$ such that
$D(E)\subset W$, where
$W=\langle\{ a, b, x, u, y\}\rangle \oplus \Lambda ^{2} E \oplus \langle\{ c\} \rangle\oplus
X \oplus U \oplus Y \oplus \langle\{f,h\}\rangle.$
\end{defi}

\begin{proposition}\label{triang}
If $B$ is a basis of $E$ and $D\in \text{Der}(\mathfrak{n})_{1}$,
then the matrix of $D$ in the basis $\mathfrak{B}(B)$ is lower triangular.
\end{proposition}

\begin{proof}
We need to check that for every element $w\in\mathfrak{B}(B) $, the coordinates of $Dw$ are zero on the basis vectors located left to $w$, according to the order in $\mathfrak{B}(B) $.
We will use the following notation:
if $v\in\n$ and $w\in\mathfrak{B}(B) $, $\lambda_w(v)$ will be the $w$ coordinate of $v$.

Since $D$ is a derivation, we know that
\[
 D(\mathfrak{n}^{\prime })\subset \mathfrak{n}^{\prime },\quad
 D([\mathfrak{n,n}^{\prime }])\subset [ \mathfrak{n,n}^{\prime}],\quad
 D(\mathfrak{z(}\mathfrak{n)})\subset \mathfrak{z(}\mathfrak{n)}.
\]
In what follows, we will omit the parenthesis and write $Dv$ for $D(v)$.

\medskip
\noindent
(1) From the definition of $\mathfrak{n}$, it is clear that
$[\mathfrak{n},\mathfrak{n}^{\prime}]=\langle h\rangle $, so $Dh=\lambda_h(Dh) h$.

\medskip
\noindent
(2) Since $c=[a,b]$ we have
\begin{eqnarray*}
Dc&=&[Da,b]+[a,Db]\\
    &=&\la_a(Da)c-(\la_u(Da)+\la_y(Da))h+\la_b(Db)c+\la_y(Db)f+\la_c(Db)h\\
    &=&(\la_a(Da)+\la_b(Db))c+\la_y(Db)f+(-\la_u(Da)-\la_y(Da)+\la_c(Db))h.
\end{eqnarray*}
This is what we need for $Dc$.

\medskip
\noindent
(3) Since $f=[x,u]=[a,y]$,  then
$[ Dx,u]+[x,Du]=[Da,y]+[a,Dy]$.
In addition,
\begin{eqnarray*}
[Dx,u] &=&\sum \la_{e_i}(Dx)u_{i}+\la_b(Dx)h+\la_{x}(Dx)f \\
\lbrack x,Du] &=&-\sum \la_{e_i}(Du)x_{i}+\la_y(Du)h+\la_{u}(Du)f \\
\lbrack Da,y] &=&\sum \la_{e_i}(Da)y_{i}+(\la_b(Da)+\la_x(Da))h+\la_a(Da)f \\
\lbrack a,Dy] &=&\lambda_b(Dy) c+\la_c(Dy)h+\la_y(Dy)f
\end{eqnarray*}
and therefore
\begin{equation}\label{by}
\lambda_b(Dy)=0\quad\text{and}\quad\la_{e_i}(Dx)=\la_{e_i}(Du)=\la_{e_i}(Da)=0\text{ for all $1\le i\le n$}.
\end{equation}
We also conclude that
\begin{eqnarray*}
Df&=&\left(\la_b(Dx)+\la_y(Du)\right)h+\left(\la_x(Dx)+\la_u(Du)\right)f\\
    &=&\left(\la_b(Da)+\la_x(Da)+\la_c(Dy)\right)h+\left(\la_a(Da)+\la_y(Dy)\right)f.
\end{eqnarray*}
This is what we need for $Df$.

\medskip
\noindent
(4) Since $D\in\text{Der}(\mathfrak{n})_{1}$, we know that
$D(e_{i})\in W$ and hence, for all $1\le i<j\le n$,
\begin{eqnarray*}
D(e_{i}\wedge e_{j})&=&[D(e_{i}),e_{j}]+[e_{i},D(e_{j})] \\
&=& -\la_{x}(De_i) x_{j}-\la_{u}(De_i)u_{j}-\la_{y}(De_i)y_{j}
\\
&& \hspace{1cm}+
\la_{x}(De_j) x_{i}+\la_{u}(De_j) u_{i}+\la_{y}(De_j) y_{i},
\end{eqnarray*}
and this is what we need to prove for $D(e_{i}\wedge e_{j})$.

\medskip
\noindent
(5) Let $1\le i\le n.$ Since $[a,e_{i}]=0$, then
$[ Da,e_{i}]+[a,D(e_{i})]=0$ and
\begin{eqnarray*}
0&=&\sum_{j} \la_{e_j}(Da) e_j\wedge e_i-\la_{x}(Da) x_{i}-\la_{y}(Da)y_{i}-\la_{u}(Da)u_{i}\\
& &\hspace{1cm}+\la_{b}(De_i)c+\la_{y}(De_i)f+\la_{c}(De_i)h,
\end{eqnarray*}
therefore
\begin{eqnarray*}
0&=& \la_{e_j}(Da) \text{ for all } 1\le i \le n,\nonumber\\
0&=& \la_{x}(Da)=\la_{y}(Da)=\la_{u}(Da),\label{t1}\\
0&=& \la_{b}(De_i)=\la_{y}(De_i)=\la_{c}(De_i),\text{ for all } 1\le i\le n.\nonumber
\end{eqnarray*}
Hence we proved, among other things, what is needed for $Da$.

\medskip
\noindent
(6) Since $[b,e_{i}]=0$, then
$[ Db,e_{i}]+[b,D(e_{i})]=0$
and taking into account that $\la_{y}(De_i)=0$ (see (5)) we obtain
\[
0=\sum_{j} \la_{e_j}(Db) e_j\wedge e_i-\la_{x}(Db) x_{i}-\la_{u}(Db)u_{i}-\la_{y}(Db)y_{i}
-\la_{a}(De_i)c+\la_{u}(De_i)h
\]
and hence
\begin{eqnarray*}
0&=& \la_{e_j}(Db) \text{ for all } 1\le j \le n,\nonumber\\
0&=& \la_{x}(Db)=\la_{y}(Db)=\la_{u}(Db)\label{buy},\\
0&=& \la_{a}(De_i)=\la_{u}(De_i),\text{ for all } 1\le i\le n.\nonumber
\end{eqnarray*}
This is almost what we we needed for $Db$.
We now combine this and results from (2) and (5)
to obtain
\[Dc =[Da,b]+[a,Db] =\big(\la_{a}(Da)+\la_{b}(Db)\big)c+\la_{c}(Db)h,\]
and hence $[Dc,b]=0$.
Since $[b,c]=0$, it follows that $[Db,c]=0$ and thus  $\la_a(Db)=0$.
This completes what we need for $Db$.

\medskip
\noindent
(7) We now consider the cases of $x$, $u$, $y$.
\begin{enumerate}[(i)]

\item We first check that the $a$ and $b$ coordinates of $Dx$, $Du$, $Dy$ are zero.

We start with $Dx$: since $[b,x]=0$, we have $[Dx,b]+[x,Db]=0$,
and since $\lambda_c([x,Db])=0$, we obtain $\lambda_c([Dx,b])=0$.
This implies that $\lambda_a(Dx)=0$.

Repeating this argument and observing that $[a,x]=0$,
we will get that $Dx$ has no $b$ coordinate.

If now we do it considering $[a,u]=0$, we will get that $Du$ has no $b$ coordinate.

We notice that the same argument,
always analyzing the $c$ coordinate,
can be repeated using $[u,b]=h$ and $[b,y]=h$
respectively to conclude that $Du$ and $Dy$ don't have $a$ coordinates.

We have already seen in  (\ref{by}) that $Dy$ doesn't have a $b$ coordinate.

\item Let us consider now the $x$ and $u$
coordinates of $Dy$: $[u,y]=0$ implies $[Du,y]+[u,Dy]=0$ and, since we know
that the $a$ coordinate of $Du$ is 0, the $f$ coordinate of $[Du,y]$ is 0, and then $[u,Dy]$ has no $f$ coordinate, which implies that the $x$ coordinate of $Dy$ is 0.

Again, the same argument considering $[x,y]=h$ leads us to conclude that the $u$ coordinate of $Dy$ is 0.

\item We consider now the $x$ coordinate of $Du$. Being $[u,y]=0$, $[Du,y]+[u,Dy]=0$.
Looking at the $h$ coordinate of this sum, we get
\[0=\la_b(Du)+\la_x(Du)-\la_b(Dy).\]
We have just seen that $\la_b(Dy)=0=\la_b(Du)$, then $\la_x(Du)=0$, as we need.

\item  Finally, we just need to prove that the $u$ and $e_i$ coordinates of $Dy$ are zero.

Since $[x,y]=h$, and recalling that
$\la_u(Dx)=0$ and $\la_{e_i}(Dx)=0$, we have
\begin{eqnarray*}
\lambda_h(Dh) h&=&D(h)=[Dx,y]+[x,Dy]\\
&=&(\la_x(Dx)h-\sum \la_{e_i}(Dy) x_i+\la_u(Dy)f +\la_y(Dy) h,
\end{eqnarray*}
so $\la_u(Dy)=0$ and $\la_{e_i}(Dy)=0$ for all $1\le i\le n$.
\end{enumerate}

\medskip
\noindent
(8) For any $1\le i\le n$, we have that
\begin{eqnarray*}
D(x_{i})&=&D([x,e_{i}])=[Dx,e_{i}]+[x,D(e_{i})]\\
&=&-\la_x(Dx) x_{i}-\la_u(Dx)u_{i}-\la_y(Dy) y_{i}+\la_u(Dx) f+\la_y(Dx) h.
\end{eqnarray*}
On the other hand,  since we know from (7.iii) that $\la_x(Du)=0$,
\begin{eqnarray*}
D(u_{i})&=&D([u,e_i])=[Du,e_{i}]+[u,D(e_{i})]
\\&=&-\lambda_u(Du) u_{i}-\la_y(Du)
y_{i}-\la_x(De_i)f-\la_b(De_i)h.
\end{eqnarray*}
Finally, having in mind that the  $x$ and $u$ coordinates of $Dy$ are 0 (see (7.ii)),
\begin{eqnarray*}
D(y_{i})&=&D([y,e_{i}])=[Dy,e_{i}]+[y,D(e_{i})]
\\ &=& -\la_y(Dy) y_{i}-\la_a(De_i) f-(\la_b(De_i)+\la_x(De_i)) h.
\end{eqnarray*}
With this we conclude the cases of $x_i$, $u_i$ and $y_i$ and the proof is complete.
\end{proof}

\begin{proposition}\label{elediag}
Let $D\in \text{Der}(\mathfrak{n})_{1}$. For each element $v\in\mathfrak B(B)$, we denote by $\lambda_v$  the diagonal coefficient  of the matrix of $D$ corresponding to the vector  $v$. Then:
\begin{enumerate}[(a)]
\item
$\lambda_{e_i}=0,\quad 1\le i\le n$.
\item
$\lambda_a=\lambda_b=\lambda_x$.
\item
$\lambda_h =\lambda_f=3\lambda_a$.
\item
$\lambda_y=\lambda_u=\lambda_c=2\lambda_a$.
\end{enumerate}
\end{proposition}
\begin{proof}
(a) is obvious from the definition. We will prove next (b), (c) and (d).
From the previous proposition we know that the matrix of
$D$ is lower triangular. Hence:
\begin{enumerate}[(1)]
\item\label{1} $[x,u]=f$ implies $ [ D(x),u]+[x,D(u)]=D(f)$ and thus
$\lambda_x+\lambda_u=\lambda_f$.

\item\label{2}
$[a,y]=f$ implies $ [ D(a),y]+[a,D(y)]=D(f)$ and thus
$\lambda_a+\lambda_y=\lambda_f.$

\item\label{3}
$[a,c]=h$ implies $ [ D(a),c]+[a,D(c)]=D(h)$  and thus
$\lambda_a+\lambda_c=\lambda_h.$

\item\label{4}
$[b,y]=h$ implies $ [ D(b),y]+[b,D(y)]=D(h)$. We know from (6)
in the proof of the previous proposition, that the $u$ coordinate of $D(b)$ is 0, and thus
$\lambda_b+\lambda_y=\lambda_h$.

\item\label{5}
$[b,u]=h$ implies $ [ D(b),u]+[b,D(u)]=D(h)$. Also from (6) in the proof of the previous proposition,
we know that the $y$ coordinate of $D(b)$ is 0, and thus
$\lambda_b+\lambda_u=\lambda_h.$

\item\label{6}
$[x,y]=h$ implies $ [ D(x),y]+[x,D(y)]=D(h)$, and thus
$\lambda_x+\lambda_y=\lambda_h.$

\item\label{7}
$[a,b]=c$ implies $ [ D(a),b]+[a,D(b)]=D(c)$, and thus
$\lambda_a+\lambda_b=\lambda_c.$

\end{enumerate}

From (\ref{3}) and (\ref{7}), we have
$2\lambda_a+\lambda_b=\lambda_h$ and combining with (\ref{5}), we obtain
\begin{equation}\label{9}
2\lambda_a=\lambda_u.
\end{equation}
Substituting in (\ref{1}),
$\lambda_x+2\lambda_a=\lambda_f.$

From (\ref{4}) and (\ref{5}), we have
\begin{equation}\label{11}
\lambda_y=\lambda_u,
\end{equation}
from (\ref{1}) and (\ref{6}), we have
\begin{equation}\label{12}
\lambda_h=\lambda_f,
\end{equation}
from (\ref{2}), (\ref{4}) and (\ref{6}),
$\lambda_a=\lambda_b=\lambda_x,$
and this proves (b).
From this and (\ref{12}) we obtain (c), that is,
$\lambda_h =\lambda_f=3\lambda_a.$

From (\ref{3}) and  (\ref{4}) it follows $\lambda_y=\lambda_c$ and from this
and  (\ref{11}) and (\ref{9}), we obtain (d).
This ends the proof of the proposition.
\end{proof}

As a consequence, we obtain the following proposition that describes the Levi decomposition of $\text{Der}(\mathfrak{n})$.

\begin{proposition}\label{thm.Levi}
Let $\text{Der}(\mathfrak{n})$ be the Lie algebra of derivations of $\n$,
and let $\text{Der}(\mathfrak{n})_{0}$ and
$\text{Der}(\mathfrak{n})_{1}$ be the Lie subalgebras of $\text{Der}(\mathfrak{n})$ defined previously.
Then:
\
\begin{enumerate}[(a)]
\item $\text{Der}(\mathfrak{n})_{0}$ is a Lie subalgebra of  $\text{Der}(\mathfrak{n})$ isomorphic to ${\mathfrak{gl}}(E)$.
\item $\text{Der}(\mathfrak{n})_{1}$ is a solvable ideal of $\text{Der}(\mathfrak{n})$.
\item $\text{Der}(\mathfrak{n})=\text{Der}(\mathfrak{n})_{0}\oplus \text{Der}(\mathfrak{n})_{1}$.
\end{enumerate}
\end{proposition}

\begin{proof}
(a) has been already discussed when we defined $\text{Der}(\mathfrak{n})_{0}$,
and (b) is a consequence of the fact that the matrix of any  $D\in\text{Der}(\mathfrak{n})_{1}$ is
lower triangular in any basis $\mathfrak B(B)$ of $\n$.

To prove (c), let us see first that the sum is direct.
If $D\in \text{Der}(\mathfrak{n})_{0}\cap \text{Der}(\mathfrak{n})_{1}$,
then $D(E)\subset E$, $D(E)\subset W$ and hence $D(E)=0$.
Since $E$ and $a,b,x,u,y$ generate $\n$ as a Lie algebra, it follows that $D=0$.

Now, we will see that $\text{Der}(\mathfrak{n})_{0}+ \text{Der}(\mathfrak{n})_{1}=\text{Der}(\mathfrak{n}).$
Given $D\in \text{Der}(\mathfrak{n})$, let $A=p_E\circ D|_{E}\in\mathfrak{gl}(E)$ where
$p_E$ the projection over $E$ with respect to the decomposition $\n=E\oplus W$.
Let $D_{0}\in \text{Der}(\mathfrak{n})_{0}$ be the derivation associated to $A$.
Since the matrix of $D_{0}$ in a basis $\mathfrak B$ is of the form (\ref{d0}),
it follows that $D_{1}=D-D_{0}\in \text{Der}(\mathfrak{n})_{1}$
This proves (c).
\end{proof}

\begin{proposition}\label{Thm.Multiplicidades}
Let $D\in \text{Der}(\mathfrak{n})$ be a diagonalizable derivation
with eigenvalues 1, 2 and 3, then the dimension of the eigenspaces are $d_1$, $d_2$ and $d_3$ respectively
(see Proposition \ref{prop.dim}). In particular, if
$\tilde\n_{1}\oplus\tilde\n_{2}\oplus\tilde\n_{3}$ is any grading of $\n(n)$, then $\dim\tilde\n_{i}=d_i$, $i=1,2,3$.
\end{proposition}

\begin{proof}
Suppose $D=D_{A}+D_{1}$ where $D_{A}\in \text{Der}(\mathfrak{n})_{0}$ and
$D_{1}\in \text{Der}(\mathfrak{n})_{1}.$
Since $D$ is diagonalizable, then $A$ is diagonalizable as well.
Then we can choose a basis $B$ of $E$
such that the matrix of $D_A$ in the  basis $\mathfrak{B}=\mathfrak{B}(B)$
is diagonal (see (\ref{d0})).
Since $D_1$ has a lower triangular matrix in the basis $\mathfrak{B}$, 
the matrix of $D$ in this basis is lower triangular.

As in Proposition \ref{elediag}, for each $v\in\mathfrak B $, we denote by $\lambda_v$
the diagonal coefficient  of the matrix of $D$ corresponding to the vector  $v$.
It is clear that $\{\lambda_v:v\in\mathfrak B \}$ are the eigenvalues of $D$ counted with multiplicity.
Now, $\lambda_v$ is either equal to 1, 2 or 3 for all $v\in\mathfrak B $.
This, together with Proposition \ref{elediag} and the shape of the matrix of $D_A$, implies that
\[
 \lambda_a=\lambda_b=\lambda_x=1,\quad \lambda_y=\lambda_u=\lambda_c=2,\quad\lambda_h =\lambda_f=3.
\]
Finally, since $D$ is a Lie algebra homomorphism, we obtain that
\[
 \lambda_{u_i}=\lambda_{e_i}+\lambda_u,\quad \lambda_{y_i}=\lambda_{e_i}+\lambda_y,\quad
 \lambda_{ x_i}=\lambda_{ e_i}+\lambda_x,\quad\lambda_{ e_i\wedge  e_j}=\lambda_{ e_i}+\lambda_{ e_j}
\]
and hence $\lambda_{e_i}=1$, $\lambda_{u_i}=3$, $\lambda_{ y_i}=3$,
$\lambda_{ y_i}=2$ and $\lambda_{ e_i\wedge  e_j}=2$.
Counting the number of eigenvalues, we obtain that the multiplicities of
 1, 2 are 3 are respectively $d_1$, $d_2$ and $d_3$.
\end{proof}

\begin{proposition}\label{teo416}
If $n\ge 17$, then
$
L(p)<2^{\dim\z}$
for all gradings $\n(n)=\n_1+\n_2+\n_3$ of $\n(n)$.
\end{proposition}

\begin{proof}
As a consequence of Proposition \ref{Thm.Multiplicidades},
the numbers $d_1$, $d_2$ and $d_3$ are independent of the grading of $\n(n)$.
Since
\[
d_{1}=n+3;\quad d_{2}=3+\frac{n(n+1)}{2};\quad d_{3}=2(n+1),\quad z=\frac{(n+4)(n+1)}{2},
\]
we need to prove that
$
L\left(\frac{(1-x)^{n+3}(1-x^2)^{\frac{n(n+1)}{2}+3}(1-x^3)^{2n+2}}{2^{\frac{(n+4)(n+1)}{2}}}\right)
< 1
$
for $n\ge 17$.

We have checked this for  $n=17,\dots, 200$ computationally, thus we must prove it for $n>200$.
We start by rearranging the factors of the polynomial in the following way:
\[
\frac{(1-x)^{n}(1-x^2)^{2n}(1-x^3)^{2n}}{2^{4n}}\cdot
\frac{(1-x^2)^{\frac{n(n-3)}{2}+3}}{{2^{\frac{n(n-3)}{2}+3}}}
\cdot 2(1-x)^{3}(1-x^3)^{2}.
\]
Since
$L\left(\frac{(1-x^2)^{\frac{n(n-3)}{2}+3}}{{2^{\frac{n(n-3)}{2}+3}}}\right)\le 1$,
$L\left(2(1-x)^{3}(1-x^3)^{2}\right)\le 64$,
and $L(pq)\le L(p)L(q)$,
we only need to show that
\[
 L(p_n)<\frac{1}{64}, \qquad \text{ where }\quad p_n(x)=\frac{(1-x)^{n}(1-x^2)^{2n}(1-x^3)^{2n}}{2^{4n}},
\]
for $n\ge 200$.
We checked computationally that
$L(p_n)<\frac{1}{2}$
 for $n=30,\dots,180$.
Now, if $n>180$, then  $n-150>30$, and arguing by induction, we obtain
\[L(p_n)\le L(p_{30})^5L(p_{n-150})\le \left(\frac{1}{2}\right)^6=\frac{1}{64}.\qedhere\]
\end{proof}



\end{document}